\documentclass[11pt]{article}
\usepackage{amsmath}
\usepackage{amssymb}
\usepackage{amsthm}
\usepackage{graphicx}
\usepackage{amsfonts}
\newtheorem{thm}{Theorem}[section]
\newtheorem{lem}[thm]{Lemma}
\newtheorem{cor}[thm]{Corollary}

\newtheorem{Def}[thm]{Definition}

\newtheorem{claim}{Claim}
\newtheorem{rem}[thm]{Remark}

\def\N{\mathcal{N}}

\def\B{\mathcal{B}}
\newcommand{\si}{\sigma}
\newcommand{\sm}{\setminus}
\newcommand{\rar}{\Rightarrow}
\newcommand{\nin}{\notin}

\numberwithin{equation}{section}

\begin{document}
\title{Homotopy Type of Neighborhood Complexes of Kneser Graphs, $KG_{2,k}$}
\author{ Nandini Nilakantan\footnote{Department of Mathematics and Statistics, IIT Kanpur, Kanpur-208016, India. nandini@iitk.ac.in.},  Anurag Singh\footnote{{Department of Mathematics and Statistics, IIT Kanpur, Kanpur-208016, India. anurags@iitk.ac.in.}}}
\date{}
\maketitle

\begin{abstract}
Schrijver identified a family of vertex critical subgraphs  of the  Kneser graphs called  the stable Kneser graphs $SG_{n,k}$. Bj\"{o}rner and de Longueville proved that the neighborhood complex of the stable Kneser graph $SG_{n,k}$ is homotopy equivalent to a $k-$sphere. 

In this article, we prove that the homotopy type of  the neighborhood complex of the Kneser graph $KG_{2,k}$ is a wedge of $(k+4)(k+1)+1$ spheres of dimension $k$. We construct a  maximal subgraph $S_{2,k}$ of $KG_{2,k}$, whose neighborhood complex is homotopy equivalent to the neighborhood complex of $SG_{2,k}$. Further, we prove that the neighborhood complex of $S_{2,k}$ deformation retracts onto the neighborhood complex of $SG_{2,k}$.
\end{abstract}

\noindent {\bf Keywords} : Hom complexes, Kneser Graphs, Discrete Morse theory.

\noindent 2000 {\it Mathematics Subject Classification} primary 05C15, secondary 57M15

\vspace{.1in}

\hrule


\section{{\bf Introduction}}
Lov\'asz in 1978, proved the Kneser Conjecture  (\cite{Lov78}), which states that the chromatic number of the Kneser graph $KG_{n,k}$ is $k+2$. In this proof, he introduced a prodsimplicial complex corresponding to a graph $G$, called the neighborhood complex of $G$, denoted by $\N(G)$. Lov\'asz showed that $\N(KG_{n,k})$ is $(k-1)$-connected. In  \cite{Sch78}, Schrijver identified a vertex critical family of subgraphs of the  Kneser graphs called the stable Kneser graphs $SG_{n,k}$ and showed that the chromatic numbers of  both $SG_{n,k}$ and $KG_{n,k}$ are the same.
In 2003, Bj\"orner and de Longueville  (\cite{BL03}) proved that the neighborhood complex of $SG_{n,k}$ is homotopy equivalent to a $k$-sphere. 

Kozlov (\cite{Koz07}, Proposition 17.28) showed that $\N(KG_{n,k})$ is homotopy equivalent to a wedge of spheres of dimension $k$.  

In this article, we identify a subgraph $S_{2,k}$ of $KG_{2,k}$, where the set of vertices of both the graphs is the same. We first show that :

 \begin{thm}\label{thm2} The neighborhood complex, $\N(S_{2,k})$, of $S_{2,k}$ collapses onto the neighborhood complex, $\N(SG_{2,k}),$ of the stable Kneser graph $SG_{2,k}$, for all $k \geq 0$.
  \end{thm}
  
  Since, $\N(SG_{2,k})$ is homotopic to $S^{k}$ (\cite{BL03}), the following is a consequence of Theorem \ref{thm2}.
  
  \begin{cor}
The neighborhood complex, $\N(S_{2,k}),$ of $S_{2,k}$ is homotopic to the $k$-sphere $S^{k},$ for all $k \geq 0$.
  \end{cor}

 The graph $S_{2,k}$ is a maximal subgraph of $KG_{2,k},$ in the sense that the addition of an extra edge to the graph $S_{2,k}$  changes the homotopy type of the corresponding neighborhood complex. Using this subgraph, we prove the main result of this article.

 \begin{thm}\label{thm3}
 The neighborhood complex, $\N(KG_{2,k}),$ of the Kneser graph $KG_{2,k}$ is  homotopic to the wedge of $k^2+5k+5$ copies of the $k$-sphere $S^{k},$ for all $k \geq 0$.
  \end{thm}

 The layout of this article is as follows: In Section 2, we present  definitions which are crucial to this article. In Section 3, we highlight some of the main results that we will be using from Discrete Morse Theory. In Section 4, we show that the neighborhood complex of $S_{2,k}$ is homotopic to the $k$-sphere {\it i.e.}, we prove Theorem \ref{thm2} and in Section 5, we prove the main result of this article, Theorem \ref{thm3}.

\section{{\bf Preliminaries}}

A {\it simple graph} is an ordered pair $G=(V,E)$, where $V$ is the set of vertices and $E  \subseteq \{ e \subseteq V  \ | \ \# (e) =2 \}$ is the set of edges of $G$.  An edge $e$ is denoted by $(v,w)$. In this article, all the graphs are simple and undirected {\it i.e.}, $(v,w)=(w,v)$. The vertices $v, w \in V$ are said to be adjacent, if $(v, w)\in E$. This is also denoted by $v \sim w$.

 A {\it graph homomorphism} from  $G$ to $H$ is a function $f : V(G) \to V(H)$ where $(v,w) \in E(G)$ implies that $(f(v),f(w)) \in E(H).$

A {\it finite abstract simplicial complex X} is a collection of finite sets where $\tau \in X$ and $\sigma \subset \tau$ implies $\sigma \in X$. The elements  of $X$ are called the  {\it simplices} of $X$. If $\sigma \in X$ and $\#(\sigma) =k+1$, then $\sigma$ is said to be $k\,dimensional$.

A {\it prodsimplicial complex} is a polyhedral complex, each of whose cells is a direct product of simplices (\cite{Koz07}).

For $v \in V(G)$, the {\it neighborhood  of $v$,} $N(v)=\{ w \in V(G) \ |  \
(v,w) \in E(G)\}$.  If $A\subset V(G)$, the neighborhood of $A$, $N(A)= \{x \in  V(G) \ | \ (x,a) \in E(G),
\,\,\forall\,\, a \in A \}$.
The {\it neighborhood complex}, $\N(G)$ of a graph $G$ is the abstract simplicial complex with vertices being the non isolated vertices of $G$ and simplices being all those subsets of $V(G)$ with a common neighbor. 

 Let  $X$ be a simplicial complex and $\tau, \sigma \in X$ such that
$\sigma \subsetneq \tau$ and  $\tau$ is the only maximal simplex in $X$ that contains $\sigma$.
  A  {\it simplicial collapse} of $X$ is the simplicial complex $Y$ obtained from $X$ by
  removing all those simplices $\gamma$  of $X$ such that
  $\sigma \subseteq \gamma \subseteq \tau$. $\sigma$ is called a {\it free face} of $\tau$ and $(\sigma, \tau)$ is called a {\it collapsible pair}. When we have a sequence of simplicial collapses from $X$ to $Z$, we say that $X$ simplicially collapses onto $Z$ and denote it by $X \searrow Z$.

Let $n \geq 1$, $k \geq 0$ and $[n]$ be the set   $ \{1,2, \ldots ,n \}$.

\begin{Def}
The Kneser graph $KG_{2,k}$ is the graph where $V(KG_{2,k}) =\{ \{i,j\} \ | \ i, j \in [k+4]  \ \text{and} \ i \neq j\}$
and  $v_1 \sim v_2$, if and only if $v_1 \cap v_2 = \emptyset$, {\it i.e.}, $(v_1,v_2) \in E(KG_{2,k}) \iff v_1 \cap v_2 =\emptyset$.

A vertex $\{i,j\}$ of $KG_{2,k}$ is said to be {\it stable} if $ j \neq i \pm 1$, with addition modulo $k+4.$ If $\{i,j\}$ is not stable, then it is said to be {\it unstable}.

The {\it stable Kneser graph} $SG_{2,k}$ is the induced subgraph of $KG_{2,k}$ each of whose vertices is stable.
\end{Def}

\begin{Def} The subgraph
$S_{2,k}$ of $KG_{2,k}$ is defined to be the graph with $V(S_{2,k})= V(KG_{2,k})$ and
$E(S_{2,k})= \{(u,v) \in E(KG_{2,k})\ | \ \text{at least one of}\ u \ \text{or} \ v\ \text{is stable} \}$.
\end{Def}

Schrijver proved that the graphs $SG_{n,k}$ are vertex critical, i.e. the chromatic number of any subgraph of $SG_{n,k}$ obtained by removing a vertex is strictly less than the chromatic number of $SG_{n,k}.$

In this article, the vertex $\{i,j\}$ of $KG_{2,k}$ will be denoted by $ij$.

\begin{Def}
Let $\alpha \in V(KG_{2,k}) $ i.e. $ \alpha =  ij$ for some $ i, j  \in [k+4]$.
The length of $\alpha$, $ l(\alpha) = |j-i| $. Clearly, if $\alpha \in V(SG_{2,k})$, then $2 \leq l(\alpha) \leq k+2$.
\end{Def}
For $\sigma =\{\alpha_1, \alpha_2, \ldots \alpha_{t}\} \in \mathcal{N}(KG_{2,k})$, let 
\begin{enumerate}
\item[$(a)$] $\sigma \oplus j = \{ \alpha_{1}  \oplus  j, \  \alpha_{2}  \oplus  j, \ \ldots , \ \alpha_{t}\oplus j\}$ and
\item[$(b)$] $\si \ominus j = \{ \alpha_{1}  \ominus  j, \ \alpha_{2}  \ominus  j,\  \ldots ,\  \alpha_{t} \ominus j \}$. 
\end{enumerate}
If $\alpha_{i} =i_1i_{2}$, then $\alpha_{i} \oplus j =(i_1 +j)(i_{2}+j)$ and $\alpha_{i} \ominus j = (i_1 -j) (i_{2}-j)$
with addition  and subtraction being modulo $k+4$.
  For example, in $\mathcal{N}(KG_{2,3})$, $67 \oplus 1 = 71 $, $23 \oplus 1 = 34$, $67 \ominus 1 = 56 $ and $12 \ominus 1 = 71=17$ (see \cite{BZ14}).

If $\sigma = \{\alpha_1, \alpha_2,\ldots,\alpha_t\} \in \mathcal{N}(KG_{2,k})$, then define 
\begin{equation}\label{eq2.1}
S_{\sigma} = \overset{t}{\underset{i=1}{\bigcup}} \ \alpha_i \ \text{and} \ \ C_\sigma = [k+4] \setminus S_{\sigma}.
\end{equation}
It is a simple observation that $C_{\sigma \setminus \{ab\}} \subseteq C_\sigma \cup \ \{a,b\}$ and $C_{\sigma \ \cup \ \{ab\}} = C_\sigma \setminus \{a,b\} .$

The following remark provides some insight into the structure of $\mathcal{N}(S_{2,k})$.

\begin{rem}\label{rem2}
\noindent
\begin{enumerate}
 \item[(i)] If $\alpha \in V(S_{2,k}) \setminus V(SG_{2,k})$, then there exists $ j \in \{0,1,2, \ldots,k+2,k+3\}$ such that $\alpha = 12 \oplus j$. Thus, $ N(\alpha)= \{u \oplus j \ | \ u \in N( 12)\} $.
\item[(ii)] If $\alpha$, $ \beta \in V(S_{2,k})$ such that $\alpha = \beta \oplus j$ for some $  j \in \{0,1,2,\ldots,k+2,k+3\}$, then $ N(\alpha)= \{u \oplus j \ | \ u \in N(\beta)\} $.
\end{enumerate}
\end{rem}

In this article, $\mathcal{N}(SG_{2,k})$, $\mathcal{N}(S_{2,k})$ and $\mathcal{N}(KG_{2,k})$ will denote the neighborhood complexes in $SG_{2,k}$, $S_{2,k}$ and $KG_{2,k}$, respectively. Further, addition and subtraction on vertices are $(\text{mod}\ k+4),$ {\it i.e.}, $30 \equiv 3(k+4) (\text{ mod } \ k+4)$ and $3(k+5) \equiv 31 (\text{ mod } \ k+4)$.

\section{{\bf Tools From Discrete Morse Theory}}

In this section, we present some of the tools that we require from Discrete Morse theory.

\begin{Def}[\cite{BZ14}]
A partial matching in a poset $(P,\text{ \textless})$ is a subset $M \subseteq P \times P$, where
\begin{enumerate}
\item[(i)] $(a,b) \in M$ implies $ b \succ a;$ {\it i.e.}, $a \ \text{\textless} \ b$ and $ \nexists $ $c \in P$ such that $ a \ \text{\textless}  \ c \ \text{\textless} \ b $ and

\item[(ii)] each $ a \in P $ belongs to at most one element in $M$.
\end{enumerate}

If $(a,b) \in M $, we denote $ a $ by $ d(b) $ and $ b$ by $ u(a) $. A partial matching on $P$ is called acyclic if there do not exist distinct $a_{i}\in P$, $1\leq i \leq  m$, $ m \geq 2 $ such that 
\begin{equation} \label{eq3.1}
 a_1 \ \prec \ u(a_1) \ \succ \ a_2 \ \prec \ u(a_2) \ \succ \ \ldots \ \prec \ u(a_m) \ \succ \ a_1 \  \ {\text is\ \ a\ \ cycle}.
\end{equation}
\end{Def}

For an acyclic matching $M$ on $P$, the unmatched elements of $P$ are called {\it critical}. If each element of $P$ belongs to a member of $M$, then $M$ is called a {\it perfect} matching on $P$.

Let $\Delta \subseteq 2^X$ and $x \in X$. Define 
\begin{equation*}
\begin{aligned}
& M_x =\{(\sigma \setminus \{x\} , \ \sigma \cup \ \{x\}) \ | \  \sigma \setminus\{ x\}, \ \sigma \cup \ \{x\} \in \Delta\} ~~and
\\
& \Delta_x= \{\sigma \in \Delta \ | \ \sigma \setminus \{x\},\ \sigma \cup  \ \{x\} \in \Delta\}.
\end{aligned}
\end{equation*}

\begin{lem}\label{pac}
The set $M_x$ is a perfect acyclic matching on $\Delta_x$.
\end{lem}
\begin{proof}
Assume that $M_x$ is not acyclic, i.e there exist pairs $(\sigma_1, \tau_1), (\sigma_2, \tau_2),$ $ \ldots , (\sigma_n, \tau_n)$ such that
$ \sigma_1 \prec \tau_1 \succ \sigma_2 \prec \tau_2 \succ \ldots \prec \tau_n \succ \sigma_1  $.
Since, $\tau_i = \sigma_i \cup \ \{x\}$ and $\tau_i \succ \sigma_{i+1}$, we have $\tau_{i+1} \subseteq \tau_i$. Therefore, $\tau_n \subseteq \tau_1 = \sigma_1 \cup \ \{x\}$. On the other hand, $\tau_n \succ \sigma_1$ implies that $\sigma_1 \subseteq \tau_n$ {\it i.e.} $\tau_1=\tau_n$, a contradiction. Hence, $M_x$ is a perfect acyclic matching on $\Delta_x$.
\end{proof}

\begin{lem}[\cite{Jon08}]\label{lem33}
Let $M'$ be an acyclic matching on $\Delta'  = \Delta \setminus \Delta_x$. Then, $M = M' \cup \ M_x$ is an acyclic matching on $\Delta$.
\end{lem}

\begin{lem}[Cluster Lemma \cite{BB11}]\label{lem34}
If $ \varphi \ : \ P \rightarrow Q$ is an order-preserving map and for each $q \in Q$, the subposet $\varphi^{-1}(q)$ carries an acyclic matching $M_q$, then ${\underset{q \ \in \ Q }{\bigsqcup}} M_q $ is an acyclic matching on P.
\end{lem}

One of the main results of Discrete Morse Theory used in this article is :

\begin{thm}[Forman \cite{For98}]\label{lem35}
Let $ \Delta $ be a simplicial complex and M be an acyclic matching on the face poset of $\Delta $. Let $c_i$ denote the number of critical $i$-dimensional cells of $\Delta$.
The space $\Delta $ is homotopy equivalent to a cell complex $ \Delta _c$ with $c_i$ cells of dimension $i$ for each $i \geq 0$, plus a single $0$-dimensional cell in the case where the empty set is also paired in the matching.
\end{thm}

This result gives the following corollaries.

\begin{cor}[\cite{BB11}]\label{rem36}
If an acyclic matching has critical cells only in a fixed dimension $i$, then $\Delta$ is homotopy equivalent to a wedge of $i$-dimensional spheres.
\end{cor}

\begin{cor}[\cite{BZ14}]\label{rem37}
If the critical cells of an acyclic matching on $\Delta$ form a subcomplex $\Delta'$ of $\Delta$, then $\Delta $ simplicially collapses to $\Delta'$, implying that $\Delta'$ is a deformation retract of $\Delta$.
\end{cor}

In this article, $\si \sm\{ ij\}$ and $\si \cup\ \{ ij\}$ will indicate that $ij \in \si$ and $ij \nin \si$, respectively. For brevity, we use $\si \sm ij$ and $\si \cup\  ij$  instead of $\si \sm\{ ij\}$ and $\si \cup\ \{ ij\}$, respectively. Further, $\{a_1 <a_2<a_3    < \ldots < a_{k}\}$ will denote a chain.

\section{{ \bf Proof of Theorem \ref{thm2}}}

Let $W_k^1=\mathcal{N}(S_{2,k})$, $W_k^{l+1}=\mathcal{N}(S_{2,k}\sm \{12,23, \ldots , l(l+1)\})$, $1\leq l \leq k+3$ and $W_k^{k+5} = \mathcal{N}(SG_{2,k})$. We first prove that $W_k^l$ collapses onto $W_k^{l+1}$, $1 \leq l \leq k+4$.
\smallskip

\begin{lem}\label{lem1}
 $ W_k^l \searrow A_{k}^l$, where $A_k^l=\{\si \in W_k^l \ | \ \ C_\sigma \neq \{l,l+1\}\}$, $1\leq l \leq k+4$.
\end{lem}
\begin{proof}
In  $S_{2,k}$, $v \sim l(l+1) \rar v \in V(SG_{2,k})$.
Observe that $A_{k}^l$ is a subcomplex of $W_k^l$.
Define $\varphi_k^l : \mathcal{F}(W_k^l) \longrightarrow \{u_k^l < v_k^l\}$  to be the map \\
\[
  \varphi_k^l(\sigma) = \left\{\def\arraystretch{1.2}%
  \begin{array}{@{}c@{\quad}l@{}}
    u_k^l & \text{if} ~~C_{\si}\neq\{l,l+1\},  \\
    v_k^l & \text{if} ~~C_{\si} = \{l,l+1\}.\\
  \end{array}\right.
\]
Let  $\tau \subseteq \si$.  If $\varphi_{k}^l(\si) = u_{k}^l$, then $C_{\si}\neq \{l,l+1\}$. Hence,  $C_{\tau}\neq \{l,l+1\}$ and $\varphi_{k}^l(\tau) = u_{k}^l$. $\varphi_k^l(\sigma)=v_k^l \rar \varphi_{k}^l(\tau) \leq \varphi_{k}^l(\si)$, always. Therefore, $\varphi_k^l$ is a poset map.

Observe that, $\si \in (\varphi_k^l)^{-1}(v_k^l) \rar \si \subseteq N(l(l+1))$.
Our aim now is to define a perfect acyclic matching on $(\varphi_k^l)^{-1}(v_k^l)$. From Corollary \ref{rem37},  we can then conclude that $W_k^l$ collapses onto $A_k^l$, thereby proving this Lemma.

If $k=1$, $N(l(l+1))=\{(l+2)(l+4)\}$. Since $C_{N(l(l+1))}\neq \{l,l+1\}$, we see that $(\varphi_1^l)^{-1}(v_1^l) $ is empty.

If $k=2$, then $(\varphi_2^l)^{-1}(v_2^l)=\{\alpha=\{(l+2)(l+4),(l+3)(l+5)\}, \{(l+2)(l+4),(l+3)(l+5),(l+2)(l+5)\}\}$. 
Thus, $M_{(l+2)(l+5)} =\{(\alpha, \alpha \cup \ (l+2)(l+5))\}$ is a perfect acyclic matching on $(\varphi_2^l)^{-1}(v_2^l)$.

Inductively, let us assume that $N_k^l$ is a perfect acyclic matching on $(\varphi_k^l)^{-1}(v_k^l)$ for all $1 \leq k \leq r-1  $, $r \geq 3$.

 Now, consider the case $k=r$. By Lemma \ref{pac}, $M_{(l+2)(l+r+3)}$ is an acyclic matching on $(\varphi_r^l)^{-1}(v_r^l)$. Let $\sigma \in (\varphi_r^l)^{-1}(v_r^l) $ and $\sigma \notin \Delta_{(l+2)(l+r+3)}$. 
 
 If $(l+2)(l+r+3) \nin \sigma$, then $C_{\si}= \{l,l+1\} \rar C_{\si \cup \ (l+2)(l+r+3)}=\{l,l+1\} \rar \si \in \Delta_{(l+2)(l+r+3)}$, a contradiction. Thus, $(l+2)(l+r+3) \in \si$, $C_{\si \sm (l+2)(l+r+3)} \neq \{l,l+1\}$ and  $(\varphi_{r}^l)^{-1}(v_{r}^l) \sm \Delta_{(l+2)(l+r+3)} =\Delta_{1}^l \cup \ \Delta_{2}^l \cup \ \Delta_{3}^l$, where

\begin{enumerate}
\item $\Delta_{1}^l=\{\si \in (\varphi_{r}^l)^{-1}(v_{r}^l)  \ | \ \{l+2,l+r+3\} \cap S_{\sigma \setminus (l+2)(l+r+3)}\ = \emptyset\}$,
\item $\Delta_{2}^l=\{\si \in (\varphi_{r}^l)^{-1}(v_{r}^l) \ | \ C_{\si \sm (l+2)(l+r+3)}=\{l,l+1,l+r+3\}\}$ and
\item $\Delta_{3}^l =\{\si \in (\varphi_{r}^l)^{-1}(v_{r}^l)  \ | \ C_{\si \sm (l+2)(l+r+3)}=\{l,l+1,l+2\}\}$.
\end{enumerate}

\begin{claim} \label{psec51} For $1 \leq l \leq r+4,$ the following maps are bijective:
\begin{enumerate}
\item[(i)] $f_l : (\varphi_{r-2}^1)^{-1}(v_{r-2}^1) \rightarrow \Delta_{1}^l$ defined by $f_l(\tau)=\{\tau\oplus l\} \cup \ \{(l+2)(l+r+3)\}$, $r >3$.
\item[(ii)] $g_l:  (\varphi_{r-1}^1)^{-1}(v_{r-1}^1) \rightarrow \Delta_{2}^l$ defined by $g_l(\tau)=\{\tau \oplus(l-1)\} \cup \ \{(l+2)(l+r+3)\}$, $r \geq 3$.
\item[(iii)] $h_l:  (\varphi_{r-1}^1)^{-1}(v_{r-1}^1) \rightarrow \Delta_{3}^l$ defined by $h_l(\tau)=\{\tau\oplus l\} \cup \ \{(l+2)(l+r+3)\}$, $r \geq 3$.
\end{enumerate}
\end{claim}
Observe that if $r=3$, then both $\Delta_{1}^l$ and $(\varphi_{1}^1)^{-1}(v_1^1)$ are empty sets.
 For any $\tau \in (\varphi_{r-2}^1)^{-1}(v_{r-2}^1)$, we see that $\tau \in W_r^l$ and $ C_\tau = \{1,2,r+3,r+4\} \rar C_{\tau \oplus l} =\{l, l+1, l+2, l+r+3\}$ (since $l+r+4 \equiv l (\text{mod} \ r+4)$). Thus, $ \{\tau \oplus l\} \cup  \ (l+2)(l+r+3) \in \Delta_{1}^l$.
$f_l$ is surjective because $\sigma \in \Delta_1^l \rar C_{\si \sm (l+2)(l+r+3)}=\{l,l+1,l+2,l+r+3\} \rar C_{(\si \sm (l+2)(l+r+3)))\ominus l} =\{1,2,r+3,r+4\} \rar (\si \sm (l+2)(l+r+3))\ominus l =\tau \in (\varphi_{r-2}^1)^{-1}(v_{r-2}^1) \rar \sigma = \{(l+2)(l+r+3)\} \cup \ \{\tau \oplus l\}$. $f$ is clearly injective and this proves Claim \ref{psec51}($(i)$). The proofs of $(ii)$ and $(iii)$ are similar.

Since, $N_{r-2}^1$ is a perfect acyclic matching  on $(\varphi_{r-2}^1)^{-1}(v_{r-2}^1)$, from Claim \ref{psec51}$(i)$, we observe that $M_1^l=\{(f_l(\tau), f_l(\tau')) \ | \ (\tau, \tau') \in N_{r-2}^1\}$ is a perfect acyclic matching on $\Delta_1^l$.

Similarly, we construct  perfect acyclic matchings $M_2^l$ and $M_3^l$ on $\Delta_2^l$ and $\Delta_3^l$ respectively.
Since, $(\varphi_r^l)^{-1}(v_r^l) = \Delta_{1}^l \cup \ \Delta_{2}^l \cup \ \Delta_{3}^l \cup \ \Delta_{(l+2)(l+r+3)}$, define 
$\theta_l : (\varphi_r^l)^{-1}(v_r^l) \longrightarrow \{ q_1^l < q_2^l < q_3^l < q_0^l \}$ by
\[
  \theta_{l}(\sigma) = \left\{\def\arraystretch{1.2}%
  \begin{array}{@{}c@{\quad}l@{}}
   q_i^l & \text{if $\sigma \in \Delta_i^l, \ 1\leq i \leq 3$ },\\
   q_{0}^{l}& \text{if $\sigma \in \Delta_{(l+2)(l+r+3)}$.}
  \end{array}\right.
\]

Let $\sigma , \tau \in (\varphi_r^l)^{-1}(v_r^l)$ and $\tau \subseteq \sigma$. 

If $\tau \in \theta_l^{-1}(q_0^l)$, then $C_{\tau \sm (l+2)(l+r+3)}=C_\tau$. Thus, $C_{\si \sm (l+2)(l+r+3)}=C_\si$ and therefore,  $\si \in \theta_l^{-1}(q_0^l)$.

  If $\sigma \in \Delta_i^l$, then by the definition of $\Delta_i^l$, $\tau \notin \Delta_j^l$ for any $j > i$. Thus, $\theta_l$ is a poset map. Hence, $M_{(l+2)(l+r+3)} \cup  \ M_1^l \cup  \ M_2^l \cup  \ M_3^l$ is a perfect acyclic matching on $(\varphi_r^l)^{-1}(v_r^l)$. 
  \end{proof}
  
We now show that for each $l \in \{1, \ldots , k+4\}$, $ A_{k}^l \searrow  W_k^{l+1}$.

\begin{lem} \label{lem2}
$ A_{k}^l \searrow  W_k^{l+1}$, $1\leq l \leq k+4$.
\end{lem}
\begin{proof}
  $S_{{2,k}}\sm \{12, 23 , \ldots , l(l+1)\}$ is a subgraph of $S_{2,k}\sm \{12, 23 , \ldots , (l-1)l\}$ and  $W_k^{l+1}$ is a subcomplex of  $W_k^l$.  $ \si \in W_k^{l+1} \rar \exists \, v \in V(S_{2,k}) \sm \{12, 23 , \ldots , l(l+1)\}$ such that $ v \in N(\si)$, {\it i.e.} $C_{\si} \neq \{l,l+1\}$, thereby showing that $W_k^{l+1}$ is a sub-complex of $A_k^l$.

\begin{claim}\label{c5.2}
$\si \in A_{k}^l \sm W_k^{l+1} \rar l(l+1) \in \si$.
\end{claim}

Assume that $l(l+1) \nin \si$.  If $\si \in A_{k}^l$, then $C_{\si}\neq \{l,l+1\}$. Hence, there exists a vertex $v \neq l(l+1)$, such that $\si \subseteq N(v)$. Further, $l(l+1) \nin \si$ implies that $ \si \in W_k^{l+1}$, a contradiction. This proves Claim \ref{c5.2}.

For $1 \leq n \leq k+1$, define $\mathcal{A}_n = \{\si \subset N(v) \subset V(S_{2,k}) \ | \ 12 \in \si, l(v) \geq n+1  \}$ and $\B_n^l = W_k^{l+1}  \cup \  \{ \si \oplus (l-1) \ | \ \si \in \mathcal{A}_n \}$. Clearly, $\B_1^l = A_k^l$ and $ \B_{k+1}^l = W_k^{l+1}$.
For $ 1 \leq i \leq k $, define $\psi_i^l : \mathcal{F}(\B_i^l) \longrightarrow \{\alpha_i^l < \beta_i^l\} $ by
\[
  \psi_i^l(\sigma) = \left\{\def\arraystretch{1.2}%
  \begin{array}{@{}c@{\quad}l@{}}
    \alpha_i^l & \text{if $\sigma \in \B_{i+1}^l ,$ }\\
    \beta_i^l & \text{otherwise.}
  \end{array}\right.
\]

Let $\sigma, \tau \in \mathcal{F}(\B_i^l)$, $\tau \subseteq \si$ and $\psi_i^l(\tau)=\beta_i^l$.  Here, $\tau \nin \mathcal{F}(\B_{i+1}^l)$, which implies that $12 \in \tau\ominus (l-1)$ and $\tau \ominus (l-1) \subset N(v)$, with $l(v)=i+1$. Now, $\tau \subset \si \rar 12\in \si\ominus (l-1) \rar \si \nin  W_k^{l+1}$. If $\si\ominus (l-1) \subset N(v)$, where $v=lm$ with $|l-m|>i+1$, then $l, m \in C_{\tau\ominus (l-1)}$. This implies that $\tau\ominus (l-1) \subset N(lm)$, a contradiction. Thus $\psi_{i}^l(\si)=\beta_{i}^l$ and therefore, $\psi_i^l$ is a poset map.

Let $i\in \{1,\dots k\}$ and  $\sigma \in  (\psi_i^l)^{-1}(\beta_i^l)$.  Here, $\sigma \in  (\psi_i^l)^{-1}(\beta_i^l)$ implies that $\si \nin \mathcal{F}(\B_{i+1}^l)$, {\it i.e.} $12 \in \si\ominus (l-1)$ and  $\si\ominus (l-1) \subset N(v)$, $l(v) =i+1$. Let $s \in [k+4]$ be the smallest integer such that $ s \in C_{\si\ominus (l-1)}$. Here, $C_{\si\ominus (l-1)} \subset C_{s}=\{s, s+1, \dots s+i+1\}$.
Define a matching $M_i^l$ on $ (\psi_i^l)^{-1}(\beta_i^l)$ as
\[
  \sigma \longrightarrow \left\{\def\arraystretch{1.2}%
  \begin{array}{@{}c@{\quad}l@{}}
    \sigma \cup  \ (l+1)(l+3) & \text{if $C_{\si\ominus (l-1)} \subset C_{3}, \ (l+1)(l+3) \nin \sigma,$ }\\
    \sigma \cup \ (l+1)(l+4) & \text{if $C_{\si\ominus (l-1)} \subset C_{4}, \ (l+1)(l+4) \nin \sigma,$}\\

    \vdots \\
    \sigma \cup  \ (l+1)(l+k+3-i) & \text{if $C_{\si\ominus (l-1)} \subset C_{k+3-i}, (l+1)(l+k+3-i) \nin \si.$}
  \end{array}\right.
\]

Let $M^{i,l}_{(l+1)r} = \{(\sigma, \sigma \cup \ (l+1)r) \ | \ C_{\sigma\ominus (l-1)} \subset C_{r}, (l+1)r \notin \sigma \}$ and $\Delta^{i,l}_{(l+1)r} = \{ \sigma \in (\psi_i^l)^{-1}(\beta_i^l)\  | \ (\sigma \setminus (l+1)r, \sigma \cup  \ (l+1)r) \in M^{i,l}_{(l+1)r} \}$, $l+3 \leq r \leq l+k+3-i$. We now prove that
 \begin{claim} \label{lem5.1}
 $(\psi_{i}^l)^{-1}(\beta_{i}^l) =\overset{l+k+3-i}{\underset{r=l+3}{\sqcup}}(\Delta_{(l+1)r}^{i,l}).$
 \end{claim}

  Clearly,  $ \overset{l+k+3-i}{\underset{r=l+3}{\sqcup}}(\Delta_{(l+1)r}^{i,l}) \subseteq(\psi_{i}^l)^{-1}(\beta_{i}^l) $. Let $\si \in (\psi_{i}^l)^{-1}(\beta_{i}^l)$. There exists a smallest $s \in \{3, \dots, k+3-i\}$ such that $C_{\si \ominus (l-1)}\subset C_{s}$. Since $s+1 <s+i+1$, $C_{\{\si\ominus (l-1)\} \sm 2(s+1)}$ (or $C_{\{\si\ominus (l-1)\} \cup \ 2(s+1)}$) $\subseteq C_{s}$. Further, $12 \in \{\si\ominus (l-1)\} \sm 2(s+1)$  (or $\{\si\ominus (l-1)\} \cup \ 2(s+1)$) shows that $(\{\si\ominus (l-1)\} \sm 2(s+1), \si\ominus (l-1))$  (or $(\si\ominus (l-1), \{\si\ominus (l-1)\} \cup \ 2(s+1))$) $\in M^{i,1}_{2(s+1)}$. Thus, $(\si \sm (l+1)(l+s), \si)$  (or $(\si, \si \cup  \ (l+1)(l+s))$) $\in M^{i,l}_{(l+1)(l+s)}$, thereby showing that $\si \in \Delta_{(l+1)(l+s)}^{i,l}$.

If $\si \in \Delta_{(l+1)(l+r)}^{i,l} \cap \Delta_{(l+1)(l+s)}^{i,l}$, $3 \leq r<s \leq k+3-i$, we see that $C_{\si\ominus (l-1)}\subset C_{s}$, $C_{r}$. $C_{\si\ominus (l-1)}\subset C_{s} \rar r \nin C_{\si\ominus (l-1)} \rar \si\ominus (l-1) \nin \Delta_{2(r+1)}^{i,1} \rar \si \nin \Delta_{(l+1)(l+r)}^{i,l}.$ Therefore, $\Delta_{(l+1)(l+r)}^{i,l} \cap \Delta_{(l+1)(l+s)}^{i,l}$ is an empty set. This completes the proof of Claim \ref{lem5.1}.

Let $M_{i}^l= \overset{l+k+3-i}{\underset{r=l+3}{\sqcup}} (M_{(l+1)r}^{i,l})$.  Lemma \ref{pac}  and Claim \ref{lem5.1} show that $M_i^l$ is a perfect acyclic matching on $(\psi_{i}^l)^{-1}(\beta_{i}^l)$, thereby showing that $\B_{i}^l\searrow \B_{i+1}^l$, from Corollary \ref{rem37}. This proves Lemma \ref{lem2}.
\end{proof}

\noindent {\bf Proof of Theorem \ref{thm2} :} 

From Lemma \ref{lem1} and Lemma \ref{lem2}, it follows that $W_k^l$ collapses to $W_k^{l+1}$ for $1\leq l \leq k+4$.  On applying this $k+4$ times, we get $\mathcal{N} (S_{2,k}) =W_{k}^{1} \searrow W_{k}^{k+5} = \mathcal{N}(SG_{2,k})$. This proves Theorem \ref{thm2}.

\section{Homotopy type of $\N(KG_{2,k})$}

In this section, $N(i(i+1))$ represents the  neighborhood  of $i(i+1)$ in the Kneser graph $KG_{2,k}$.  Let $X_{i,i+1}= \{\si \subseteq N(i(i+1)) \ | \ \si \nin \mathcal{N}(S_{2,k})\}$. For any $v \in V(SG_{2,k})$, $N(v)$ is the same in both $S_{2,k}$ and $KG_{2,k}$. We observe the following:
\begin{rem}\label{rem6.1}
\begin{enumerate}
\item[$(i)$] $\si \in X_{i,i+1} \Leftrightarrow C_\si =\{i,i+1\}$ and $j(j+1)\in \si$, $j \neq i \in [k+4]$.
\item[$(ii)$] $\mathcal{N} (KG_{2,k})=\mathcal{N}(S_{2,k}) \ \bigsqcup \ ( \bigsqcup\limits_{i=1 }^{k+4} X_{i,i+1})$.
\end{enumerate}
\end{rem}

Let $i \in [k+4]$, $I_{i}=[k+4] \sm \{i-1, i, i+1\}$ and $I_{j}^{i}= [k+4] \sm \{i, i+1, j-1, j, j+1\}$ for each $j \in I_{i}$.
If $t \in I_{j}^{i}$ and $\mathcal{I}_{t} =I_{j}^{i} \sm \{t,t+1,\ldots, k+4\}$, define
\begin{equation}\label{e6.1}
\begin{aligned}
& E_j^i= \{\sigma \in X_{i,i+1} \ | \ j(j+1) \in \sigma ~\text{and}~ \ s(s+1) \notin \sigma, ~\forall~ s \in I_i \sm \{j,j+1 \dots, k+4\}\},
\\
& F_{j}^{i} =\{\si \in E_{j}^{i} \ | \ js \in \si ~\text{and}~ C_{\si \sm js}=\{i, i+1, s\}, \forall ~s \in I_{j}^{i}\} ~~and
\\
& F_{j,t}^{i} =\{\si \in E_{j}^{i}\ | \ \forall~ r \in \mathcal{I}_{t}, ~ jr\in \si, r \in C_{\si\sm jr}~\text{and either}~jt \nin \si ~\text{or}~ C_{\si \sm jt}=\{i, i+1\}\}.
\end{aligned}
\end{equation}

\begin{lem}\label{lem61}
Let $i \in [k+4]$, $j \in I_i$ and $t \in I_j^i$. Then,
\begin{enumerate}

   \item[(i)] $X_{i,i+1}= \bigsqcup\limits_{j \in I_i } E_j^i$,
   \item[(ii)] $E_j^i= \bigsqcup\limits_{t \in I_j^i} F_{j,t}^i \sqcup F_{j}^i$ and
   \item[(iii)] $\si \in F_{j}^{i} \rar \si = \{(j-1)(j+1), j(j+1)\} \cup \{jr \ | \ r \in I_{j}^{i}\}$.

  \end{enumerate}
\end{lem}
\begin{proof}
\begin{enumerate}

\item[$(i)$] Let $\si \in X_{i,i+1}$. From Remark \ref{rem6.1}$(i)$, $C_\si = \{i,i+1\}$ and $s(s+1) \in \si$ for some $s \in [k+4]$. Let $l=\text{min}\{ s \in [k+4]\ | \ s(s+1) \in \si \}.$ Since this set is nonempty and $l \in I_{i}$, we get  $\si \in E_l^i $. Thus, $X_{i,i+1} \subseteq \bigsqcup\limits_{j \in I_i } E_j^i. $

\item[$(ii)$]  Let $\si \in E^i_j,$ $T=\{l \in I^i_j \ | \ l \in C_{\si \sm jl} \}$ and $S \subseteq T$ be the set $\{r \in I^i_j\sm \{k+4\} \ | \ s \in C_{\si \sm js} \ \forall \ s \in \mathcal{I}_{r+1} \}$. If $S$ is an empty set and $t_1= \text{min} \{l \ | \ l \in I_j^i \}$, then $\mathcal{I}_{t_1}$ is an empty set. If $t_1\in  C_{\si \sm jt_1},$ then $t_1\in S,$ a contradiction. Thus, either $jt_1 \notin \si$ or $C_{\si \sm jt_1}=C_\si,$ which shows that $\si \in F_{j,t_1}^i$. 

Now, consider the case when $S$ is non empty. If $T=I^i_j$, then $ s \in C_{\si \sm js} \ \forall \  s \in I^i_j $ which implies that $\si \in F^i_j.$ If $T \neq I^i_j$, then $r \in S \rar s\in S, \ \forall \ s <r, s \in I_j^i$ {\it i.e.} $t \in I_j^i\sm T \rar t>l, \ \forall \ l \in S.$ Choose $t^{'}= \text{min} \{l \ | \ l \in I_j^i\sm T\}$ and $t_2 = \text{max} \{t \in I_j^i \ | \ t \in S\}$.  Here, $t^{'} \nin T\rar jt^{'} \nin \si \text{ or } t^{'} \nin C_{\si \sm jt^{'}} \rar \si \in F_{j,t^{'}}^i$. $T \neq I^i_j \rar t_2 < \text{max}\{l \ | \ l \in I^i_j\}$. Thus, $E_j^i = \bigsqcup\limits_{t \in I_j^i} F_{j,t}^i \sqcup F_{j}^i.$

\item[$(iii)$]  From Equation \eqref{e6.1}, $\si \in F_{j}^{i}$ implies that $s \in C_{\si \sm js}$ for each $ s \in I_{j}^{i}$. Further, $\si \in F_{j}^{i}$ shows that $\si \in E_{j}^{i}.$ This implies that $j(j+1) \in \si$ and $(j-1)j \nin \si$. Since $C_{\si} =\{i,i+1\}$, we see that $j-1 \in S_{\si}$, which implies that $(j-1)(j+1)$ is the only possible element of $\si$ containing $j-1$.
\end{enumerate}
\end{proof}

Let $P_k$ denote the face poset of $\mathcal{N} (KG_{2,k})$.

\begin{lem}\label{lem62} For $i \in [k+4]$, $j \in I_i$ and $t \in I_j^i$, the following are poset maps:
\begin{enumerate}
  \item[(i)] $ \varphi : P_k \longrightarrow \{a_1  >  a_2 >  a_3 > \ldots > a_{k+4} > a\}$ defined by 
  \[
  \varphi^{-1}(x) = \left\{\def\arraystretch{1.2}%
  \begin{array}{@{}c@{\quad}l@{}}
    X_{i,i+1}& \text{if $x=a_{i} \ i \in \{1, \ldots , k+4\},$ }\\
    \mathcal{F}(\mathcal{N}(S_{2,k}))& \text{if $x=a$.}
  \end{array}\right.
\]
  
  \item[(ii)] $\varphi_i:X_{i,i+1} \longrightarrow \{b_{1}^i > b_{2}^i >\ldots> b_{i-2}^i  > b_{i+2}^i > b_{i+3}^i > \ldots > b_{k+4}^i \}$ defined by $\varphi_i^{-1}(b_j^i)=E_j^i,$ for $j \in I_i$.

  \item[(iii)]  $\varphi_{i,j} : E_{j}^{i} \longrightarrow J$, where $J=\{d_{j,t}^{i} \ | \ t \in I_{j}^{i} \} \cup  \ \{d^i_j\}$ with $d_{j,t}^{i}  > d_{j,t'}^{i} >d^i_j$ for $t <t'$, defined by
   \[
  \varphi_{i,j}(\si) = \left\{\def\arraystretch{1.2}%
  \begin{array}{@{}c@{\quad}l@{}}
    d_{j,t}^{i}& \text{if $\si \in F_{j,t}^{i}, \ t \in I_{j}^{i},$ }\\
    d_{j}^{i}& \text{if $\si \in F_{j}^{i}$.}
  \end{array}\right.
\] 
  \end{enumerate}
\end{lem}

\begin{proof}\begin{enumerate}
\item[$(i)$] Let $\tau \subseteq \si$. From Remark \ref{rem6.1}, $\tau \in X_{i,i+1}$ implies that $C_\tau =\{i,i+1\}$ and $j(j+1) \in \tau $, for some $j \neq i$. Therefore, $ C_\si \subseteq C_\tau$ implying that $ \si \in X_{i,i+1}$, thereby proving $(i)$.

\item[$(ii)$] Let $\tau \subseteq \si$ and $\tau \in E_{s}^{i}$. Assume that $\si \in E_r^i$,  where $r \in I_i$ and $r >s$. Here, for each $j \in I_i \sm \{r,r+1 \dots, k+4\}$, we see that $j(j+1) \notin \si$ which implies that $s (s+1) \nin \tau$, a contradiction. Thus, $\varphi_i(\si) \geq \varphi_i(\tau)$.

\item[$(iii)$] Assume that $\varphi_{i,j} (\tau)= d_{j,r}^{i}$ and $\varphi_{i,j} (\si)= d_{j,s}^{i}$, $r<s$, where $\tau \subseteq \si$.
Here $\tau \in F_{j,r}^{i}.$  From Equation \eqref{e6.1},  either $jr \nin \tau$ or $C_{\tau \sm jr} =\{i, i+1\}$. Thus, $ r \in S_{\tau}$.  If $jr \nin \tau$, then from Equation \eqref{e6.1}, $r <s \rar jr  \in \si$ and $C_{\si \sm jr} =\{i,i+1,r\}$. Therefore, $r \in C_{\tau \sm  jr}=C_{\tau}$, a contradiction.
If $C_{\tau \sm jr} =\{i, i+1\}$, then $C_{\si \sm jr} =\{i, i+1\}$, a contradiction. Thus, $r >s$. Further if $\varphi_{i,j} (\si)= d_{j}^{i}$, then $\tau \in E^i_j\rar C_\tau=\{i,i+1\}, j(j+1)\in \tau$. So, $\tau \subseteq \si\rar \tau =\si$. Hence, $\varphi_{i,j}$ is a poset map.
\end{enumerate}
\end{proof}

\noindent We now prove Theorem \ref{thm3}.

\begin{proof}
Let $i \in [k+4]$, $j \in I_i$ and $t \in I_j^i$. Then,
\begin{enumerate}
\item[$(i)$] $E^i_j $ has an acyclic matching with one critical $k$-cell.

Let $\si \in F_{j,t}^i$. If $jt \in \si$, then $C_{\si \sm jt}=\{i,i+1\}$ implies that $\si \sm jt \in F_{j,t}^i$. If $jt \nin \si$, then $ j,t\neq i,i+1$ implies that $C_{\si\cup \ jt}=\{i,i+1\}$. Thus, $C_{\{\si \cup \ jt\}\sm jt}=\{i,i+1\}$ which implies that $\si \cup jt \in F_{j,t}^i$. Hence, $F_{j,t}^i = \Delta_{j,t}^i=\{\si \in F_{j,t}^i \ | \ (\si \sm jt, \si \cup \ jt) \in M_{j,t}^i \}$, where $M_{jt}^i = \{(\si \sm jt, \si \cup \ jt) \ | \ \si \sm jt, \si \cup \ jt \in F_{j,t}^i \}$. By Lemma \ref{pac}, $M_{jt}^i$ is a perfect acyclic matching on $\Delta_{j,t}^i=F_{j,t}^i = \varphi_{i,j}^{-1}(d_{i,j}^t)$. From the lemmas \ref{lem62}$(iii)$ and \ref{lem34}, $M_j^i= \bigsqcup\limits_{t\in I_j^i } M_{jt}^i$ is an acyclic matching on $\bigsqcup\limits_{t \in I_j^i} F_{j,t}^i \sqcup F_{j}^i=E_j^i$ with $F_{j}^i$ as the set of critical cells. Since $F_{j}^i$ contains exactly one $k$-cell from  Lemma \ref{lem61}$(iii)$, we see that $M_j^i$ is an acyclic partial matching on
$E_j^i=\varphi_{i}^{-1}(b_j^i)$ with exactly one critical $k$-cell.

\item[$(ii)$] $X_{i,i+1}$ has an acyclic matching with $(k+1)$ critical $k$-cells.

By Lemmas \ref{lem62}$(ii)$ and \ref{lem34}, $M_i= \bigsqcup\limits_{j \in I_i} M_j^i$ is an acyclic matching on $X_{i,i+1} = \bigsqcup\limits_{j \in I_i} E_j^i = \varphi^{-1}(a_i)$ with $|I_i|=(k+1)$ critical $k$-cells.

\item[$(iii)$] $P_k$ has an acyclic matching with $(k+4)(k+1)+1$ critical $k$-cells. In \cite{BL03}, Bj\"orner and de Longueville proved that $\mathcal{N}(SG_{n,k})$ is homotopy equivalent to a $k$-sphere.
It can be shown that $\mathcal{F}(\mathcal{N}(S_{2,k}))$ has an acyclic matching with exactly one critical cell of dimension $k$ (the details of this matching are available with the authors).   Now, using Theorem
 \ref{thm2}, we have an acyclic matching, say R, on $\mathcal{F}(\mathcal{N}(S_{2,k})) = \varphi^{-1}(a)$ with one critical cell of dimension $k$ . Thus, by Lemmas \ref{lem62}$(i)$ and \ref{lem34}, $\{\bigsqcup\limits_{i= 1}^{k+4} M_i\} \bigsqcup R$ is an acyclic partial matching on $\{\bigsqcup\limits_{i= 1}^{k+4} \varphi^{-1}(a_i)\} \bigsqcup \varphi^{-1}(a)=P_k$ with $(k+4)(k+1)+1 = k^2+5k+5$ critical cells of dimension $k$.
\end{enumerate}
The proof  of Theorem \ref{thm3} now follows from Corollary \ref{rem36}.
\end{proof}

\section{Acknowledgement}
The authors would like to thank the referee for having provided several relevant comments. 

\bibliographystyle{alpha}

\end{document}